\documentclass[11pt,leqno]{article}
\usepackage{graphicx, amsfonts, amsthm, amsxtra, amssymb, verbatim, makeidx}
\usepackage{subeqnarray, relsize}
\usepackage[mathscr]{euscript}
\usepackage{hyperref}
\makeindex
\makeglossary
\begin{document}
\newtheorem{theo}{Theorem}
\newtheorem{exam}{Example}
\newtheorem{coro}{Corollary}
\newtheorem{defi}{Definition}
\newtheorem{prob}{Problem}
\newtheorem{lemm}{Lemma}
\newtheorem{prop}{Proposition}
\newtheorem{rem}{Remark}
\newtheorem{conj}{Conjecture}
\newtheorem{calc}{}

%------Chapter on algebraic cycles-----------------------
\def\gru{\mu} %the group of d-th roots of unity 
\def\pg{{ \sf S}}               %permutation group
\def\TS{{\mathlarger{\bf T}}}                %Tangent space
\def\NB{{\mathlarger{\bf N}}}
\def\group{{\sf G}}
\def\NLL{{\rm NL}}   %Noether-Lefschetz locus

\def\plc{{ Z_\infty}}    %a section of X by a projective space of dimension n/2+1
\def\pola{{u}}      %polarization
\newcommand\licy[1]{{\mathbb P}^{#1}} %linear projective cycle
\newcommand\aoc[1]{Z^{#1}}     % Aoki-Shioda cycle
\def\HL{{\rm Ho}}     %Hodge locus
\def\NLL{{\rm NL}}   %Noether-Lefschetz locus

%----------------General Math Notations-------------------------------------
\def\Z{\mathbb{Z}}                   %Integer  numbers
\def\Q{\mathbb{Q}}                   %Rational  numbers
\def\C{\mathbb{C}}                   %Complex numbers
\def\N{\mathbb{N}}                   %natural numbers
\def\uhp{{\mathbb H}}                %upper half plane
\def\A{\mathbb{A}}                   %affine space C^n
\def\dR{{\rm dR}}                    %The subindex dR standing for de Rham cohomology.
\def\F{{\cal F}}                     %A foliation
\def\Sp{{\rm Sp}}                    %Symplectic group
\def\Gm{\mathbb{G}_m}                 %The multiplicative group
\def\Ga{\mathbb{G}_a}                 %The additive  group
\def\Tr{{\rm Tr}}                      %Trace map in Algebraic de Rham cohomology
\def\tr{{{\mathsf t}{\mathsf r}}}                 %Transposition of matrices
\def\spec{{\rm Spec}}            %The spectrume
\def\ker{{\rm ker}}              %kernel
\def\GL{{\rm GL}}                %The liner group
\def\ker{{\rm ker}}              %kernel
\def\coker{{\rm coker}}          %cokernel
\def\im{{\rm Im}}               %Image
\def\coim{{\rm Coim}}            %coimage
\def\p{{\sf  p}}
\def\U{{\cal U}}   %covering

\def\weig{{\nu}}
\def\r{{ r}}                       %dimension of the moduli space of hypersurfaces.
%----------------Gauss-Manin Connection in disguise---------------------
\def\k{{\sf k}}                     %Arbitrary field
\def\ring{{\sf R}}                   %A ring
\def\X{{\sf X}}                      %families of varieties
\def\Ua{{   L}}                      %families of affine varieties
\def\T{{\sf T}}                      %Moduli of enhanced  varieties
\def\asone{{\sf A}}                  %affine space A^{n+1} defined over a ring

\def\Ts{{\sf S}}
\def\cmv{{\sf M}}                    %Classical moduli of varieties
\def\BG{{\sf G}}                       %Borel Algebraic Group
\def\podu{{\sf pd}}                   %Poincare dual
\def\ped{{\sf U}}                    %Period Domain
\def\per{{\bf  P}}                   %period matrix or fundamental system of GMC
\def\gm{{  A}}                    %Gauss-Manin connection
\def\gma{{\sf  B}}                   %Gauss-Manin connection
\def\ben{{\sf b}}                    %Betti number

\def\Rav{{\mathfrak M }}                     % Space of modular vector fields
\def\Ram{{\mathfrak C}}                     % Space of constant vector fields
\def\Rap{{\mathfrak G}}                     % Space of vector fields arising from group  action. 

\def\mov{{\sf  m}}                    %Fixed dimension of the cohomology
\def\Yuk{{\sf C}}                     %Yukawa and generalizations 
\def\Ra{{\sf R}}                      %Ramanujan type vector field
\def\hn{{ h}}                         %Hodge numbers
\def\cpe{{\sf C}}                     %Constant periods
\def\g{{\sf g}}                       %An element of the Borel group
\def\t{{\sf t}}                       %An element \T
\def\pedo{{\sf  \Pi}}                  %Period domain before discrete group action

\def\Der{{\rm Der}}                   %Derivations
\def\MMF{{\sf MF}}                    %Moduli of modular foliations
\def\codim{{\rm codim}}                %codimension
\def\dim{{\rm    dim}}                %dimension
\def\Lie{{\rm Lie}}                   %Lie algebra of a group.
\def\gg{{\mathfrak g}}                %elements of a Lie group

\def\u{{\sf u}}                       %An element of the period domain

\def\imh{{  \Psi}}                 %intersection matrix in homology
\def\imc{{  \Phi }}                  %intersection matrix in cohomology
\def\stab{{\rm Stab }}               %stablizer 
\def\Vec{{\rm Vec}}                 %space of vector fields
\def\prim{{\rm  0}}                  %primitive cohomology

\def\Fg{{\sf F}}     %Genus g topological partition function
\def\hol{{\rm hol}}  %holomorphic
\def\non{{\rm non}}  %non-holomorphic
\def\alg{{\rm alg}}  %algebraic
\def\tra{{\rm tra}}  %transcendental

\def\bcov{{\rm \O_\T}}       %The ring of modular-type functions

\def\leaves{{\cal L}}        %space of leaves  

\def\cat{{\cal A}}              %category
\def\im{{\rm Im}}               %Image

%%%%%%%%%%%%%%%%%%%%%%%%%%%%%%%%%%%%%%%%%%%%%%%%%%%%%%%%%%%%%%%%%
\def\pn{{\sf p}}              %Periods of Hodge cycles
\def\Pic{{\rm Pic}}           %Picard group 
\def\free{{\rm free}}         %Free part
\def \NS{{\rm NS}}    %Neron-Severi group 
\def\tor{{\rm tor}}
\def\codmod{{\xi}}    %Codimension of the Hodge loci

%---------------OLD Notation-------------------
\def\GM{{\rm GM}}

\def\perr{{\sf q}}        %period matrix.....
\def\perdo{{\cal K}}   %period domain
\def\sfl{{\mathrm F}} %Space of filtrations
\def\sp{{\mathbb S}}  %Sphere

\newcommand\diff[1]{\frac{d #1}{dz}} %Differential operator
\def\End{{\rm End}}              %Endomorphism group

\def\sing{{\rm Sing}}            %The set of singularities
\def\cha{{\rm char}}             %Charracteristic
\def\Gal{{\rm Gal}}              %The Galois group
\def\jacob{{\rm jacob}}          %the Jacobian ideal
\def\tjurina{{\rm tjurina}}      %the tjurina ideal
\newcommand\Pn[1]{\mathbb{P}^{#1}}   %Projective space of dimension #1
\def\P{\mathbb{P}}
\def\Ff{\mathbb{F}}                  %Finite field

\def\O{{\cal O}}                     %ring of integers of a number field

\def\ring{{\mathsf R}}                         %A ring
\def\R{\mathbb{R}}                   %real numbers

\newcommand\ep[1]{e^{\frac{2\pi i}{#1}}}% unipotent numbers
\newcommand\HH[2]{H^{#2}(#1)}        %Hodge structures
\def\Mat{{\rm Mat}}              %Matrices
\newcommand{\mat}[4]{
     \begin{pmatrix}
            #1 & #2 \\
            #3 & #4
       \end{pmatrix}
    }                                %two by two matrices
\newcommand{\matt}[2]{
     \begin{pmatrix}                 % one by two matrix
            #1   \\
            #2
       \end{pmatrix}
    }
\def\cl{{\rm cl}}                %Chern class

\def\hc{{\mathsf H}}                 %The set of Hodge cycles.
\def\Hb{{\cal H}}                    %Hodge bundle
\def\pese{{\sf P}}                  %Period set

\def\PP{\tilde{\cal P}}              %the period domain/ discrete group
\def\K{{\mathbb K}}                  %Field representing R or C

\def\M{{\cal M}}
\def\RR{{\cal R}}
\newcommand\Hi[1]{\mathbb{P}^{#1}_\infty}%the hyperplane at infinity
\def\pt{\mathbb{C}[t]}               %Polynomials in t
\def\gr{{\rm Gr}}                %graded pieces
\def\Im{{\rm Im}}                %imaginary
\def\Re{{\rm Re}}                %Real
\def\depth{{\rm depth}}
\newcommand\SL[2]{{\rm SL}(#1, #2)}    %SL(2,Z)
\newcommand\PSL[2]{{\rm PSL}(#1, #2)}  %PSL(2,Z)
\def\Resi{{\rm Resi}}              %Residue

\def\L{{\cal L}}                     %The moduli of polarized lattices in a
                                     %fixed vector spaces.
\def\Aut{{\rm Aut}}              %Automorphism group of a vectorspace
\def\any{R}                          %Any subring of the field of complex
                                     %numbers.
\newcommand\ovl[1]{\overline{#1}}    %Conjugation of #1.

\newcommand\mf[2]{{M}^{#1}_{#2}}     %New modular functions
\newcommand\mfn[2]{{\tilde M}^{#1}_{#2}}     %New modular functions

\newcommand\bn[2]{\binom{#1}{#2}}    %Binomial
\def\ja{{\rm j}}                 %j of a two by two matrix
\def\Sc{\mathsf{S}}                  %Simple cycles
\newcommand\es[1]{g_{#1}}            %Eisenstein series
\newcommand\V{{\mathsf V}}           %Milnor vector space
\newcommand\WW{{\mathsf W}}          %Similar to Milnor vector space
\newcommand\Ss{{\cal O}}             %Structural sheaf
\def\rank{{\rm rank}}                %rank of a module
\def\Dif{{\cal D}}                   %Differentials
\def\gcd{{\rm gcd}}                  %greatest common divisor
\def\zedi{{\rm ZD}}                  %zero divisors of a module
\def\BM{{\mathsf H}}                 %Brieskorn module
\def\plf{{\sf pl}}                             %Picard-Lefschetz formula
\def\sgn{{\rm sgn}}                      %sign
\def\diag{{\rm diag}}                   %diagonal matrix
\def\hodge{{\rm Hodge}}
\def\HF{{ F}}                                %The hodge filtration of the brieskon module
\def\WF{{ W}}                               %The weight filtration of the brieskon module
\def\HV{{\sf HV}}                                %humbert variety
\def\pol{{\rm pole}}                               %pole divisor
\def\bafi{{\sf r}}
\def\id{{\rm id}}                               %identity
\def\gms{{\sf M}}                           %Gauss-Manin system
\def\Iso{{\rm Iso}}                           %Gauss-Manin system

\def\hl{{\rm L}}    %holomorphic limit
\def\imF{{\rm F}}
\def\imG{{\rm G}}

\def\cf{r}   %The coefficients of the linear cycles.
\def\cm{\checkmark}
\def\MI{{\cal M}}
\def\se{{\sf s}}

\begin{center}
{\LARGE\bf Special components of Noether-Lefschetz loci
}
\\
\vspace{.25in} {\large {\sc Hossein Movasati}}
\footnote{
Instituto de Matem\'atica Pura e Aplicada, IMPA, Estrada Dona Castorina, 110, 22460-320, Rio de Janeiro, RJ, Brazil,
{\tt www.impa.br/$\sim$ hossein, hossein@impa.br.}}
\end{center}
\begin{abstract}
We take a sum $C_1+\cf C_2,\ \cf\in\Q$ of a line $C_1$ and a complete intersection curve $C_2$ of type 
$(3,3)$ inside a smooth surface of 
degree $8$ and with $C_1\cap C_2=\emptyset$.  We gather evidences to the fact that for all except a finite number of $\cf$, 
the Noether-Lefschetz loci attached to the cohomology classes of $C_1+\cf C_2$ are distinct $31$ codimensional subvarieties intersecting 
each other  in a $32$ codimensional subvariety of the ambient space.  
The maximum  codimension for components of the Noether-Lefschetz locus in this case is $35$, and hence,  
we provide a conjectural description of a counterexample to a conjecture of J. Harris. 
The methods used in this paper also  produce in a rigorous way an  infinite number of general components passing through the point representing the Fermat surface of degree $\leq 9$, and many non-reduced components for such degrees.  
\end{abstract}

\section{Introduction}
In the parameter space $\T_{\rm full}$ of smooth surfaces of degree $d\geq 4$ in $\P^3$ the Noether-Lefschetz locus $\NLL_d$ is a union of enumerable subvarieties of $\T_{\rm full}$ and its points parameterize surfaces with Picard number $\geq 2$. A component of $\NLL_d$  of codimension equal to (resp. strictly less than) $h^{20}=\binom{d-1}{3}$ is called general (resp. special). It is known that general components are 
dense in $\T_{\rm full}$ in  both usual and Zariski topology, see \cite{CHM88}, \cite[\S 5.3.4]{vo03} and \cite{CL1991}, and special  components of codimension $d-3$  and $2d-7, \ d\geq 5$ are unique and parameterize respectively surfaces with a line and conic, see \cite{green1988, green1989, voisin1988, voisin89}. This implies that for $d=5$ we have only two special components. J. Harris in 1980's conjectured that the number of special components must be  
finite. C. Voisin in \cite{voisin1991} found counterexamples to this for a large $d$, however, the conjecture in lower degrees remains open. In \cite{voisin90} it is  proved that for $d=6,7$ the number of reduced special components 
is  finite, and so, it is expected that Harris' conjecture is true in these cases.  However, for $d=8$ it is widely open.  In this article we describe a conjectural description of an infinite number of reduced special components of $\NLL_8$.

Let $X_0$ be a smooth surface in $\Pn 3$ of degree $d\geq 4$. We assume that the Picard number $\rho(X_0)$ of $X_0$ is bigger than or equal to 
$3$, and  hence,  $X_0$ has two curves $C_1$ and $C_2$ whose cohomology classes are linearly independent in the second primitive cohomology of $X_0$. We consider a one dimension family $[C_1]+\cf [C_2]\in H_2(X_0,\Q),\ \ \cf\in\Q$ and the corresponding family of Noether-Lefschetz loci  $V_{[C_1]+\cf [C_2]}$ inside the parameter space $\T_{\rm full}$ of smooth surfaces in $\P^3$, 
see \S\ref{CaporasoHuybrechtDebarre2019} for the definition. It is  equipped with an analytic scheme  structure and its underlying 
analytic variety is a union of branches of $\NLL_d$ near $0\in\T$.  In the present paper 
we are looking for  a special pencil of Noether-Lefschetz locus $V_{[C_1]+\cf[C_2]}$. 
\begin{defi}\rm 
 We say that $V_{[C_1]+\cf[C_2]}$ is a special pencil if 1. for all $\cf\in\Q$,
 $\codim \TS_0V_{[C_1]+\cf[C_2]}<h^{20}:=\binom{d-1}{3}$, 2.    
%we have $\codim \left(\TS_0V_{[C_1]}\cap \TS_0V_{[C_2]}\right)<\codim \TS_0V_{[C_1]+\cf[C_2]}$ and 
there is no inclusion between the tangent spaces $\TS_0V_{[C_1]}$ and $\TS_0V_{[C_2]}$ and 
3. for all $\cf\in\Q$ except a finite number,  $V_{[C_1]+\cf[C_2]}$ is smooth  as an analytic scheme (and hence reduced). If instead of the last property, the $N$-th infinitesimal  
Noether-Lefschetz locus $V^N_{[C_1]+\cf[C_2]}$ is the $N$-jet of a smooth  variety then we call it an $N$-th infinitesimal special pencil. 
If at least for one $\cf\in\Q$ we have  $\codim \TS_0V_{[C_1]+\cf[C_2]}=h^{20}(X_0)$ and the condition 2 as above is satisfied then  
$V_{[C_1]+\cf[C_2]}$ is automatically smooth  and we call it  a general pencil.  
\end{defi}
If  a special pencil exists, it gives us an infinite number of special reduced components of $\NLL_d$ passing through a point, and hence, a counterexample to Harris' conjecture.  We focus on the following class of examples. 
Let $d, d_1,d_2,s_1,s_2,m_1,m_2$ be integers with 
\begin{equation}
\label{bolso2019}
1\leq d_1\leq d_2\leq \frac{d}{2},\ 1\leq s_1,s_2\leq \frac{d}{2},\ \  \ 0\leq m_1\leq \min \{d_1, s_1\}\  \ 0\leq m_2\leq \min\{d_2, s_2\}.
\end{equation}
Let also $f=f_1f_3f_5f_7 +f_2f_4f_6f_8\in\C[x]_d:=\C[x_0,x_1,x_2,x_3]_d$ with  
$$
f_1\in\C[x]_{m_1},\ \ f_2\in \C[x]_{m_2},\ \ f_3\in \C[x]_{d_1-m_1},  f_4\in \C[x]_{d_2-m_2}, \ \
$$
$$
f_5\in \C[x]_{s_1-m_1},  f_6\in \C[x]_{s_2-m_2},\ \  f_7\in \C[x]_{d-d_1-s_1+m_1},  f_8\in \C[x]_{d-d_2-s_2+m_2}.
$$
We consider the surface $X_0\in\P^3$ given by $f=0$ and two algebraic curves 
\begin{eqnarray}
C_1&:&  f_1f_3=f_2f_4=0,\\
C_2&:&  f_1f_5=f_2f_6=0. 
\end{eqnarray}
Our main example is the Fermat surface given by $f=x_0^d+x_1^d+x_2^d+x_3^d$ and 
\begin{eqnarray}
& &  f_1:=\prod_{i=0}^{m_1-1}(x_0-\zeta_{2d}^{2i+1}x_1),\ \ 
f_3:=\prod_{i=m_1}^{d_1-1}(x_0-\zeta_{2d}^{2i+1}x_1),\ \ 
f_5:=\prod_{i=d_1}^{d_1+s_1-m_1 -1}(x_0-\zeta_{2d}^{2i+1}x_1),\\
& &  f_2:=\prod_{i=0}^{m_2-1}(x_2-\zeta_{2d}^{2i+1}x_3),\ \ 
f_4:=\prod_{i=m_2}^{d_2-1}(x_2-\zeta_{2d}^{2i+1}x_3),\ \
f_6:=\prod_{i=d_2}^{d_2+s_2-m_2 -1}(x_2-\zeta_{2d}^{2i+1}x_3),
\end{eqnarray}
and $f_7,f_8$ are the rest of the factors in the factorization of $x_0^d+x_1^d$ and $x_2^d+x_3^d$.  
In this paper we prove the following. 
\begin{theo}
 \label{main}
 Let us consider the Fermat surface of degree $d=4,5,6,7,8$ and a choice of  integers in 
 \eqref{bolso2019}  except the case 
 \begin{equation}
 \label{forastf!!}
 d=8,\ \  \{(d_1,d_2), (s_1,s_2) \}=\{(3,3),  (1,1)\} \ \  (m_1,m_2)=(0,0). 
 \end{equation}
 Assume that $\codim \TS_0V_{[C_1]+\cf[C_2]}<\binom{d-1}{3}$ and  there is no inclusion between 
 $\TS_0V_{[C_1]}$ and $\TS_0V_{[C_2]}$. Moreover
  \begin{equation}
\label{alexis2019-2}
\cf:=\frac{\cf_2}{\cf_1},\ \ \cf_1,\cf_2,\in\Z,\ \  1\leq \cf_1\leq 10, \ \  0 \leq |\cf_2|\leq 10
\end{equation}
The Noether-Lefschetz locus $V_{[C_1]+\cf[C_2]}$ with 
\begin{equation}
\label{alexis2019}
3\leq \cf_1,\  \ \hbox{ or }\ \   3\leq |\cf_2|,
\end{equation}
%\begin{equation}
%\label{alexis2019}
%\cf:=\frac{\cf_2}{\cf_1},\ \ \cf_1,\cf_2,\in\Z,\ \  3\leq \cf_1,\ \ 3\leq |\cf_2|\leq 10,
%\end{equation}
is singular as an analytic scheme (as  an analytic variety this means that either it is singular at the Fermat point  $0$ or 
its defining ideal  is non-reduced).
\end{theo}
For further non-reducedness statements see  \cite[Proposition 1]{mclean2005}, \cite[Theorem 1.2]{Dan2017}, 
\cite[Theorem 18.3]{ho13}.
The number of cases such that the hypothesis of Theorem \ref{main} is satisfied is the difference of $\#$  with the sum of  `General'  and  `Inclusion'  in Table \ref{jinbabak2019}. For instance for $d=5$ we have $10=61-(47+4)$ such cases. 
The upper bound for $\cf_1, |\cf_2|$ in  \eqref{alexis2019-2} is due to our computational methods, and so, the above theorem suggests that 
$V_{[C_1]+\cf [C_2]}$ is not a special pencil except for \eqref{forastf!!}. 
In this exceptional case we have all the properties of a special pencil except the last one. 
We expect  this case  provides a special pencil. In order to provide evidences
for this missing property we consider the following deformation of the Fermat surface: 
%given in 
%the affine chart $x_0=1$:
\begin{equation}
\label{danithegamer2019}
X_t: \ \ x_0^8+x_1^8+x_2^8+x_3^8-\sum t_i x^i=0,
\end{equation}
where the sum runs through the following collection of $32$ monomials 
$$
x_1^6x_3^2, 
x_1^6x_2x_3,
   x_1^5x_3^3,
   x_0x_1^4x_3^3,
   x_1^6x_2^2,
   x_1^5x_2x_3^2,
   x_0x_1^4x_2x_3^2,
   x_1^4x_3^4,
   x_0x_1^3x_3^4,
   x_0^2x_1^2x_3^4,
   $$
   $$
   x_1^5x_2^2x_3,
   x_0x_1^4x_2^2x_3,
   x_1^4x_2x_3^3,
   x_0x_1^3x_2x_3^3,
   x_0^2x_1^2x_2x_3^3,
   x_0^3x_3^5,
   x_0x_1^2x_2x_3^4, 
   x_1^3x_2x_3^4,
   x_1^4x_2^2x_3^2,
   x_0x_1^3x_2^2x_3^2,
   $$
   $$
   x_0^2x_1^2x_2^2x_3^2,
   x_0^2x_1x_2x_3^4, 
   x_0x_1^2x_2^2x_3^3,
   x_1^3x_2^2x_3^3,
   x_0^2x_2^2x_3^4,
   x_1^2x_2x_3^5,
   x_0x_1x_2^3x_3^3,
   x_0x_1^5x_3^2,
   x_0^2x_1^3x_3^3,
   x_0^3x_1x_3^4,
   $$
   $$
   x_0^2x_1x_2^2x_3^3,
   x_0^2x_2^3x_3^3. 
$$
This deformation is chosen in such a way that $V_{[C_1]}\cap V_{[C_2]}=\{0\}$, see \S\ref{02june2019}.   
We prove that 
\begin{theo}
\label{koshti2019}
 For \eqref{forastf!!} the infinitesimal Noether-Lefschetz locus $V^5_{[C_1]+\cf [C_2]}$ in the parameter space of
 the deformation \eqref{danithegamer2019} and with  \eqref{alexis2019-2} 
 is the $5$-jet of a smooth  variety. 
\end{theo} 
% [1]:
%    x(1)^6*x(3)^2
% [2]:
%    x(1)^6*x(2)*x(3)
% [3]:
%    x(1)^5*x(3)^3
% [4]:
%    x(1)^4*x(3)^3
% [5]:
%    x(1)^6*x(2)^2
% [6]:
%    x(1)^5*x(2)*x(3)^2
% [7]:
%    x(1)^4*x(2)*x(3)^2
% [8]:
%    x(1)^4*x(3)^4
% [9]:
%    x(1)^3*x(3)^4
% [10]:
%    x(1)^2*x(3)^4
% [11]:
%    x(1)^5*x(2)^2*x(3)
% [12]:
%    x(1)^4*x(2)^2*x(3)
% [13]:
%    x(1)^4*x(2)*x(3)^3
% [14]:
%    x(1)^3*x(2)*x(3)^3
% [15]:
%    x(1)^2*x(2)*x(3)^3
% [16]:
%    x(3)^5
% [17]:
%    x(1)^2*x(2)*x(3)^4
% [18]:
%    x(1)^3*x(2)*x(3)^4
% [19]:
%    x(1)^4*x(2)^2*x(3)^2
% [20]:
%    x(1)^3*x(2)^2*x(3)^2
% [21]:
%    x(1)^2*x(2)^2*x(3)^2
% [22]:
%    x(1)*x(2)*x(3)^4
% [23]:
%    x(1)^2*x(2)^2*x(3)^3
% [24]:
%    x(1)^3*x(2)^2*x(3)^3
% [25]:
%    x(2)^2*x(3)^4
% [26]:
%    x(1)^2*x(2)*x(3)^5
% [27]:
%    x(1)*x(2)^3*x(3)^3
% [28]:
%    x(1)^5*x(3)^2
% [29]:
%    x(1)^3*x(3)^3
% [30]:
%    x(1)*x(3)^4
% [31]:
%    x(1)*x(2)^2*x(3)^3
% [32]:
%    x(2)^3*x(3)^3
In our way to prove Theorem \ref{main} and Theorem \ref{koshti2019}, we have found many general pencils.
By definition  members of such a pencil are reduced and smooth at the Fermat point $0$. 
\begin{theo}
\label{main2}
 For $d=4,5,6,7,8,9$,  a general pencil $V_{[C_1]+\cf [C_2]}$ exists and the number of such pencils 
 are listed under the column `General' in Table \ref{jinbabak2019}. For $d=10,11$ general pencils do not exist.  
\end{theo}
In order to prove Theorem \ref{main}, Theorem \ref{koshti2019} and  Theorem \ref{main2} we have produced 
Table \ref{jinbabak2019} which contains more data than what is  announced in these theorems. 
Let us explain this table for the row $d=6$. 
The number of pairs $(C_1,C_2)$ in this case is $355=212+15+79+49$. Among these we have $212$ general pencils. 
The number of pairs with an inclusion between $\TS_0V_{[C_i]},\ \ i=1,2$ is $15$.
In the remaining cases we have analyzed the algebraic cycles 
$\cf_1 C_1+\cf_2 C_2$ with \eqref{alexis2019-2}.
Note that if we set $\cf:=\frac{\cf_1}{\cf_2}$ then we have $V_{\cf_1 [C_1]+\cf_2 [C_2] }=V_{[C_1]+\cf [C_2]}$ as the Noether-Lefschetz loci is unchanged if we multiply
the algebraic cycle by a rational number. In these cases  we have analyzed 
$V_{[C_1]+\cf [C_2]}$  in the parameter space which is described in \S\ref{02june2019}. 
For $N=2,3,4,5,6$ the number of 
pairs $(C_1,C_2)$ such that at least for one $(\cf_1,\cf_2)$, $V_{[C_1]+\cf [C_2]}$  is not  
$N$-smooth but
it is $M$-smooth for all $(\cf_1,\cf_2)$ as above and $M<N$, is respectively 
$79$, $49$, $0$, $0$ and $0$. The only exceptional case is $d=8$ and the two cases \eqref{forastf!!}.  In these  
cases $V_{[C_1]+\cf [C_2]}$ is $5$-smooth for all 
$(\cf_1,\cf_2)$ in \eqref{alexis2019-2} and the author  was not able to verify the $6$-smoothness. 
The number $299$ under NT refers to the number of cases such that 
at least for one $\cf$ as in \eqref{alexis2019-2}, $\TS_0V_{[C_1]+\cf[C_2]}$
is not transversal to the smaller deformation space described in \S\ref{02june2019}.
The numbers in Table \ref{jinbabak2019} are hyperlinked to the author's webpage in which 
the reader can find the computer produced data. Except for the  last column, the data is 
organized in the following way. It is a list of lists of the form:
{\tiny
\begin{verbatim}
 [i]:                          
   [1]:
      d_1,d_2
   [2]:
      s_1,s_2
   [3]:
      m_1,m_2
   [4]:
      a_1,a_2,a_3,a_4
   [5]:
      [1]:
         r_1,r_2
      [2]:
          ...
\end{verbatim}  
}where 
\begin{equation}
\label{solopeidando2019}
(a_1,a_2,a_3,a_4)=
(\codim(\TS_0V_{[C_1]} ),\  \codim(\TS_0V_{[C_2]} ), \codim(\TS_0V_{[C_1]+\cf [C_2]} ) ,
\codim(\TS_0V_{[C_1]}\cap \TS_0V_{[C_2]}  ),
\end{equation}
and  the fifth item is the list of all $(\cf_1,\cf_2)$ such that 
$V^N_{\cf_1[C_1]+\cf_2[C_2]}$ is the $N$-jet of a smooth  variety (it does not exist for the second and third  columns under `General' and `Inclusion'). In the case of last column, the fifth item of the data consists of all 
$(\cf_1,\cf_2)$ with ${\rm gcd}(\cf_1,\cf_2)=1$ such that $\TS_0V_{[C_1]+\cf[C_2]}$
is not transversal to the smaller deformation space described in \S\ref{02june2019}.

\begin{table}
\begin{center}
\begin{tabular}{|c|c|c|c|c|c|c|c|c|c|c|c|}
\hline
$d$  & $\#$ & General &     Inclusion   &  $N=2$  & $N=3$ &  $N=4$ & $N=5$ &  $N=6$ & $N\geq 7$ & NT\\ \hline 

$4$ &  $61$ & \href{http://w3.impa.br/~hossein/WikiHossein/files/Singular%20Codes/2019_06_Special_Components/2019_06_d=4_general_components.txt }
{$54$}  
&  
\href{http://w3.impa.br/~hossein/WikiHossein/files/Singular%20Codes/2019_06_Special_Components/2019_06_d=4_inclusion.txt}
{$7$}
& $0$ &  $0$ &  $0$ & $0$ & $0$ & $0$ &
\href{http://w3.impa.br/~hossein/WikiHossein/files/Singular%20Codes/2019_06_Special_Components/2019_06_d=4_NonTransversal.txt
}
{$7$}
\\ \hline

$5$ & $61$ &
\href{http://w3.impa.br/~hossein/WikiHossein/files/Singular%20Codes/2019_06_Special_Components/2019_06_d=5_general_components.txt
}
{$47$}
&    
\href{http://w3.impa.br/~hossein/WikiHossein/files/Singular%20Codes/2019_06_Special_Components/2019_06_d=5_inclusion.txt
}
{$4$}
&   
$0$&  
\href{http://w3.impa.br/~hossein/WikiHossein/files/Singular%20Codes/2019_06_Special_Components/2019_06_d=5_3Smooth.txt
}
{$5$}
&  
$0$
& $0$ &  
\href{http://w3.impa.br/~hossein/WikiHossein/files/Singular%20Codes/2019_06_Special_Components/2019_06_d=5_6Smooth.txt
}
{$5$}
&
$0$
& 

\href{http://w3.impa.br/~hossein/WikiHossein/files/Singular%20Codes/2019_06_Special_Components/2019_06_d=5_NonTransversal.txt
}
{$39$}

%\\
%$5$ & 
%
%\href{http://w3.impa.br/~hossein/WikiHossein/files/Singular%20Codes/2019_06_Special_Components/%2019_06_d=5_general_components.txt
%}
%{$47$}
%& 
%
%\href{http://w3.impa.br/~hossein/WikiHossein/files/Singular%20Codes/2019_06_Special_Components/2019_06_d=5_inclusion.txt
%}
%{$4$}
%
%& $0$&  
%\href{http://w3.impa.br/~hossein/WikiHossein/files/Singular%20Codes/2019_06_d=5_3Smooth_Full.txt
%}
%{$10$}
%&  $0$ & $0$ &  $0$ 
\\ \hline
$6$ & $355$ &
\href{http://w3.impa.br/~hossein/WikiHossein/files/Singular%20Codes/2019_06_Special_Components/2019_06_d=6_general_components.txt
}
{$212$}
&  
\href{http://w3.impa.br/~hossein/WikiHossein/files/Singular%20Codes/2019_06_Special_Components/2019_06_d=6_inclusion.txt
}
{$15$}
&   
\href{http://w3.impa.br/~hossein/WikiHossein/files/Singular%20Codes/2019_06_Special_Components/2019_06_d=6_2Smooth.txt
}
{$79$*}
&  
\href{http://w3.impa.br/~hossein/WikiHossein/files/Singular%20Codes/2019_06_Special_Components/2019_06_d=6_3Smooth.txt
}
{$49$}
&  $0$ & $0$ & $0$  & $0$ &

\href{http://w3.impa.br/~hossein/WikiHossein/files/Singular%20Codes/2019_06_Special_Components/2019_06_d=6_NonTransversal.txt
}
{$299$}
\\ \hline
$7$ & $355$ &
\href{http://w3.impa.br/~hossein/WikiHossein/files/Singular%20Codes/2019_06_Special_Components/2019_06_d=7_general_components.txt 
}
{$66$}
&  
\href{http://w3.impa.br/~hossein/WikiHossein/files/Singular%20Codes/2019_06_Special_Components/2019_06_d=7_inclusion.txt
}
{$17$}
&   
\href{ http://w3.impa.br/~hossein/WikiHossein/files/Singular%20Codes/2019_06_Special_Components/2019_06_d=7_2Smooth.txt
 }
{$229$}
&  
\href{http://w3.impa.br/~hossein/WikiHossein/files/Singular%20Codes/2019_06_Special_Components/2019_06_d=7_3Smooth.txt
}
{$35$}
&  
\href{http://w3.impa.br/~hossein/WikiHossein/files/Singular%20Codes/2019_06_Special_Components/2019_06_d=7_4Smooth.txt
}
{$8$}
& $0$ & $0$  & $0$ &

\href{http://w3.impa.br/~hossein/WikiHossein/files/Singular%20Codes/2019_06_Special_Components/2019_06_d=7_NonTransversal.txt
}
{$342$}

\\ \hline
$8$
&
\href{http://w3.impa.br/~hossein/WikiHossein/files/Singular%20Codes/2019_06_Special_Components/2019_06_d=8_rest.txt
}
{$1220$}+113 
 & 
\href{http://w3.impa.br/~hossein/WikiHossein/files/Singular%20Codes/2019_06_Special_Components/2019_06_d=8_general_components.txt
}
{$113$}
&
\href{http://w3.impa.br/~hossein/WikiHossein/files/Singular%20Codes/2019_06_Special_Components/2019_06_d=8_inclusion.txt
}
{$45$}
&   

\href{http://w3.impa.br/~hossein/WikiHossein/files/Singular%20Codes/2019_06_Special_Components/2019_06_d=8_2Smooth.txt
}
{$1155$}

&
\href{http://w3.impa.br/~hossein/WikiHossein/files/Singular%20Codes/2019_06_Special_Components/2019_06_d=8_3Smooth.txt
}
{$18$}
&
\href{http://w3.impa.br/~hossein/WikiHossein/files/Singular%20Codes/2019_06_Special_Components/2019_06_d=8_4Smooth.txt
}
{$0$}  
&
 
$0$
&
$?$
&
\href{http://w3.impa.br/~hossein/WikiHossein/files/Singular%20Codes/2019_06_Special_Components/2019_06_d=8_5Smooth_Picture_Shell.txt
}
{$? $} 

& 

\href{http://w3.impa.br/~hossein/WikiHossein/files/Singular%20Codes/2019_06_Special_Components/2019_06_d=8_NonTransversal.txt
}
{$1319$}

\\ \hline \hline
 $9$
 &
\href{http://w3.impa.br/~hossein/WikiHossein/files/Singular%20Codes/2019_06_Special_Components/2019_06_d=9_rest.txt
}
{$1314$}+19 
&  
\href{http://w3.impa.br/~hossein/WikiHossein/files/Singular%20Codes/2019_06_Special_Components/2019_06_d=9_general_components.txt
}
{$19$}
&

&    

&

&

&  
& 
& 
&

\\ \hline
$10$   &
\href{http://w3.impa.br/~hossein/WikiHossein/files/Singular%20Codes/2019_06_Special_Components/2019_06_d=10_rest.txt
}
{$3873$} 
& 

{$0$}
&

&    

&

&

&  
& 
& 
&

\\ \hline
$11$ &
\href{http://w3.impa.br/~hossein/WikiHossein/files/Singular%20Codes/2019_06_Special_Components/2019_06_d=11_rest.txt
}
{$3873$} 
& 

{$0$}
&

&    

&

&

&  
& 
& 
&

\\ \hline

\end{tabular}
\end{center}
\caption{Number of general/special components etc.  }
\label{jinbabak2019}
\end{table}
I would like to thank Roberto Villaflor, Ananayo Dan and Emre Sert\"oz 
for many useful conversations in the early stages of the present article. 
%This   motivated the author to prepare a polished v

\section{Preliminaries}
\label{CaporasoHuybrechtDebarre2019}
For $N=d, d-4$  let  
\begin{equation}
\label{21oct2014}
I_N:=\left \{ (i_0,i_1,i_2,i_3)\in {\mathbb Z}^{4}\Big| 0\leq i_e\leq d-2, \ \ i_0+i_1+i_2+i_{3}=N\right\},     
\end{equation}
and for $N=2d-4$ we define $\check I_N$ as above but with the  stronger condition $i_0+i_1=d-2,\  i_2+i_3=d-2$. 
Let  $B_0, B_1$ be 
subsets of $\{\zeta\in \C  | \zeta^d+1=0\}$ with cardinalities $d_1,d_2$, respectively. 
For  $i\in \check I_{(\frac{n}{2}+1)d-n-2}$ we define the number 
\begin{equation}
\label{10a2017}
p_i:=\left(\sum_{\zeta\in B_{0}}\zeta^{i_0+1}\right)\cdot
\left(\sum_{\zeta\in B_{1}}\zeta^{i_{2}+1}\right).
\end{equation}
For any other \(i\) which is not in the set \(\check I_{(\frac{n}{2}+1)d-n-2}\), \(\pn_i\) by definition is zero. The complete intersection algebraic cycle 
$$
C: \prod_{\zeta\in B_0}(x_0-\zeta x_1)=\prod_{\zeta\in B_1}(x_2-\zeta x_2)
$$
has the periods 
\begin{equation}
\pn_i([C]):=
\label{churrasco2019}
\mathlarger{\mathlarger{\int}}_{C}
{\rm Residue}\left(  
\frac{x_0^{i_0}x_1^{i_1}x_2^{i_2} x_{3}^{i_3}\cdot  \sum_{i=0}^{3}(-1)^ix_i\widehat{dx_i}}
     {( x_0^{d}+x_1^{d}+x_2^{d}+x_{3}^d)^{2}}\right)
=\frac{2\pi \sqrt{-1}}{d^2} p_i
\end{equation}
see \cite[Theorem 1]{roberto}.  Let \([\pn_{i+j}]\)  be the matrix whose rows and columns are indexed by 
\(i\in I_{d-4}\) and \(j\in I_d\),  respectively, and in its \((i,j)\) entry we have \(\pn_{i+j}\). 

We consider the family of surfaces $X_t\subset \P^3$ given by the homogeneous polynomial:
\begin{equation}
\label{15dec2016}
f_t:=x_0^{d}+x_1^{d}+x_2^d+x_{3}^d-\sum_{j\in I_d}t_j x^j=0,\ \ 
\end{equation}
where $t=(t_j)_{j\in I_d}\in(\T,0)$. 
In a Zariski neighborhood of the Fermat variety, and up  to linear transformations of $\Pn {3}$,  
every  surface  can be written in this format. More precisely, the derivative of the canonical map 
$i: {\rm PGL}(4,\C)\times \T\to \T_{\rm full}$ at $({\rm identity},0)$ is an isomorphism, and hence, $i$ is etale at this point.  
By definition $\T$ is a Zariski open 
subset of the vector space $\C[x^{I_d}]$ generated by $x^i,\ \ i\in I_d$ and it parameterizes smooth surfaces.  Therefore, $f\in \T$ parametrizes a surface  given by 
$x_0^d+x_1^d+x_2^d+x^d_{3}+f=0$. In this way, $\TS_0\T=\C[x^{I_d}]$. Any statement  on Noether-Lefschetz locus for the full parameter space $\T_{\rm full}$ which appears 
in the present article
follows from the same statement for $\T$, and from now on, we will only consider $\T$.

A cycle $\delta_0\in H_2(X_0,\Q)$ satisfying 
$$
\mathlarger{\int}_{\delta_0}\omega=0,\ \ \forall  \omega\in H^0(X_0, \Omega_{X_0}^2)
$$ 
is called a Hodge cycle.  Let  $\omega_1,\omega_2,\cdots,\omega_a,\ \ a=h^{20}(X)$ be sections of the bundle 
$H^0(X, \Omega^2_{X_t}),\ t\in(\T,0)$ such that  
they form a basis at each fiber and $\delta_t\in H_n(X_t,\Q)$ be  the monodromy/parallel
transport of $\delta_0$ to $X_t$, see \cite[\S 5.3.2]{vo03}. 
The analytic space  $V_{\delta_0}$ with 
\begin{equation}
\label{10maio16}
\O_{V_{\delta_0}}:=\O_{\T,0}\Bigg/\left\langle \mathlarger{\int}_{\delta_t}\omega_1, 
\mathlarger{\int}_{\delta_t}\omega_2, \cdots, \mathlarger{\int}_{\delta_t}\omega_a \right\rangle,
%\subset \spec (\O_{\T,0}),
\end{equation}
is called the Noether-Lefschetz locus passing through $0$ and corresponding to $\delta_0$.
It might be non-reduced, see for instance \cite[Exercise 2, page 154]{vo03}. The tangent space of the Noether-Lefschetz locus at the Fermat point is given by 
\begin{equation}
\label{16.06.2019}
\TS_0V_{\delta_0}=\ker([\pn_{i+j}(\delta)]):=\left\{ \sum_{i\in I_d} {v_i}x^i \Big |   [v_i][\pn_{i+j}(\delta)]^\tr=0 \right\},
\end{equation}
where $\pn_i(\delta_0):=\int_{\delta_0}\omega_i,\ \ i\in I_{2d-4}$ are periods of $\delta_0$. 
This follows from infinitesimal variation of Hodge structures introduced in \cite{CGGH1983}. For an easy proof of this see \cite[\S 16.5]{ho13}.

Let $\MI_{\T,0}$ be the maximal ideal of $\O_{\T,0}$, that is, the set of germs of holomorphic functions in $(\T,0)$
vanishing at $0$. The $N$-th order infinitesimal scheme  
$V^N_{\delta_0}$ is the induced scheme  by \eqref{10maio16} in the infinitesimal scheme $\T^N:=
\spec(\O_{\T,0}/\MI^{N+1}_{\T,0})$. We denote by $\X^N/\T^N$ the $N$-th order 
infinitesimal deformation of $X_0$ induced by $\X/\T$. 
Let $\cl(Z_0)\in H^n_\dR(X_0)$ be the class of a divisor  $Z_0$ in $X_0$. Let us consider the  Gauss-Manin connection 
$$
\nabla: H^2_\dR(\X/\T)\to \Omega_\T^1\otimes_{\O_\T} 
H^2_\dR(\X/\T).
$$
It induces a connection in $H^2_\dR(\X^N/\T^N)$ which we call it again the Gauss-Manin connection.  
There is a unique section $\se$ 
of  $H^2_\dR(\X^N/\T^N)$ such that $\nabla(\se)=0$ and $\se_0=\cl(Z_0)$. This is called
the horizontal extension of $\cl(Z_0)$ or a flat section of the cohomology bundle.
An equivalent definition for $V^N_{[Z_0]}$ is as follows.
\begin{defi}\rm
\label{kisin2018} 
 The  infinitesimal Noether-Lefschetz locus $V_{[Z_0]}^N$ is a subscheme of $\T^N$ 
 given by the conditions
\begin{eqnarray}
 \label{17/11/2018-1}
& & \nabla(\se)=0,\\ \label{17/11/2018-2}
&  & \se \in F^{1}H^2_\dR(\X^N/\T^N),\\  \label{17/11/2018-3}
& & \se_0=\cl(Z_0).
\end{eqnarray}
\end{defi}
\begin{defi}\rm
 We say that $V_{[Z_0]}$ is $N$-smooth if $V^N_{[Z_0]}$  the $N$-jet of a smooth  variety at $0$.
 For a more computational and differential geometric approach to smoothness see \cite[\S 18.5]{ho13}.
\end{defi}

\section{General components}
Using the following proposition we can produce many examples of general pencils. 
\begin{prop}
\label{danieltoulouse2019}
 Let $X_0$ be a smooth surface of degree $d$. Assume that $X_0$  has two Hodge cycles 
 $\delta_1$ and $\delta_2$ such that 
 \begin{enumerate}
 \item
 $V_{\delta_1}$ is very general, in the sense that $\codim \TS_0V_{\delta_1}=h^{20}:=\binom{d-1}{3}$.  
 \item 
  $\TS_0V_{\delta_1}\not\subset \TS_0V_{\delta_2}\not= \TS_0\T$
 \end{enumerate}
 Then the Noether-Lefschetz loci 
 $V_{\delta_1+\cf\delta_2}$ for all $\cf\in\Q$ except a finite number of them,  are set theoretically different, smooth, reduced  and very general.
\end{prop}
\begin{proof}
The function $\cf\mapsto \codim (\TS_0V_{\delta_1+\cf\delta_2} )$ is lower semi-continuous and it reaches its maximum at $\cf=0$. This implies that for all except a finite number of $\cf\in\Q$ we have 
$\codim (\TS_0V_{\delta_1+\cf\delta_2} )= \codim (\TS_0V_{\delta_1} )$.  For two rational numbers $\cf_1,\cf_2$ with $\cf_1\not= \cf_2$ we have 
$$
\TS_0V_{\delta_1+\cf_1\delta_2}\cap \TS_0V_{\delta_1+\cf_2\delta_2}=\TS_0V_{\delta_1}\cap \TS_0V_{\delta_2}
$$
and the codimension of this vector space is bigger than $h^{20}$. This follows from our hypothesis 
$\TS_0V_{\delta_1}\not\subset \TS_0V_{\delta_2}$ and $\codim \TS_0V_{\delta_1}= h^{20}$. 
%Therefore, there cannot be inclusion between
%$\TS_0V_{\delta_1+\cf\delta_2}$ for all $x$ except a finite number as above. 
This implies that 
the vector spaces 
$\TS_0V_{\delta_1+\cf\delta_2}$ form a pencil  with the axis $\TS_0V_{\delta_1}\cap \TS_0V_{\delta_2}$.
Since $\codim (\TS_0V_{\delta_1+x\delta_2})$ is also the number of equations defining $ V_{\delta_1+\cf\delta_2}$, the statement follows.  
\end{proof}
\begin{proof}{(of Theorem \ref{main2})}
We just need to check the hypothesis of Proposition \ref{danieltoulouse2019} for all 
pairs $(C_1,C_2)$
In Table \ref{jinbabak2019} under the column `General' we have the number of general pencils 
among all pencils described in the Introduction. Clicking at each number the reader can find 
the list of such pencils. This includes the case 
\begin{equation}
\label{khaste2019}
d_1=d_2=s_1=s_2=\left[\frac{d-1}{2}\right],\ \ \ \
%\left\{
% \begin{array}{ll}
% \frac{d}{2} &    d  \hbox{even}\\
% \frac{d-1}{2} &    d  \hbox{ odd}
% \end{array}\right.,    \ \ 
m_1=m_2=0.  
\end{equation}
\end{proof}
\begin{rem}\rm
%The proof of Theorem \ref{main2}  suggests that Noether-Lefschetz locus 
%$V_{[C_1]+\cf [C_2]}$ with \eqref{khaste2019} is general for all $d\geq 4$. 
%However, this is not true for $d=9,8$.
%For instance, 
For $d=9,10$  the Noether-Lefschetz locus $V_{[C_1]+\cf [C_2]}$ with \eqref{khaste2019} is not a general pencil. 
We have $(a_1,a_2,a_3,a_4)=(46,46,50,72), h^{20}=56$ and for $d=10$ we have  
$(a_1,a_2,a_3,a_4)=(62,62,80,114)$ and $h^{20}=84$,
where $a_i$'s are defined in \eqref{solopeidando2019}. 
%$a=\frac{d}{2}$ (resp. $\frac{d-1}{2}$) for $d$ an even (resp. odd) number.  
%For disjoint complete intersection curves $C_1,C_2$ of type $(a,a)$ inside a smooth surface $X_0$ of degree $d$, $V_{[C_1]+\cf[C_2]},\ \ \cf\in\Q$ is a general pencil.  
\end{rem}

\section{Deformation space}
\label{02june2019}
Let us take two Hodge cycles 
$\delta_1,\ \delta_2\in H_2(X_0,\Z)$. We would like to compute a vector space $W\subset\TS_0\T$ such that 
$\TS_0\T$ is a direct sum of $\TS_0V_{\delta_1}\cap\TS_0V_{\delta_2}$ and $W$, and $\TS_0V_{\delta_1 +\cf \delta_2}$ intersects $W$ transversely, that is, the codimension of $\TS_0V_{\delta_1 +\cf \delta_2}$ in $\TS_0\T$ is equal to the codimension of 
$\TS_0V_{\delta_1 +\cf \delta_2}\cap W$ in $W$. For this we consider the vertical   concatenation  $A$ of
$[\pn_{i+j}(\delta_1)]$ and $[\pn_{i+j}(\delta_1)]$.  Its kernel is 
$ \ker([\pn_{i+j}(\delta_1)])\cap \ker [\pn_{i+j}(\delta_2)]$. We compute a $a\times a$ minor $B$ of $A$ such that $\det(B)\not=0$ and $a$ 
is the rank of $A$. Let $I^*$ be the set of row indices of $B$.    
We also check that the submatrix of $[\pn_{i+j}(\delta_1+\cf\delta_2)]$, with rows indexed by $I^*$ and all columns has the same rank as 
$[\pn_{i+j}(\delta_1+\cf\delta_2)]$.  This implies that  $\TS_0V_{\delta_1 +\cf \delta_2}$ intersects $W$ transversely, 
where the vector space $W$ is generated by
monomials $x^i, i\in I^*$. In the new deformation space 
\begin{equation}
\label{clarakaren2019}
 X_t:\ \ x_0^d+x_1^d+x_2^d+x_3^d-\sum_{i\in I^*}t_ix^i=0,\ \ t:=(t_i,\ i\in I^*)\in\C^{\#I^*}, 
\end{equation}
we have $\TS_0V_{\delta_1}\cap\TS_0V_{\delta_2}=\{0\}$ and  
$\TS_0V_{\delta_1 +\cf \delta_2}$ form a pencil of vector spaces intersecting each other at $\{0\}$.
The procedure {\tt DeformSpace} is dedicated to the computation of the deformation space in \eqref{clarakaren2019}.
%Now, let $\delta_k=[C_k]$ be the support of a complete intersection algebraic cycle. 
%It is well-known that  the locus $V_[C]$ is smooth and reduced and it is equal to the image $V_C$ of the projection map 
%${\rm hilb}(X,C)\to {\rm hilb}(X)$ near $X$.  

\section{The creation of a formula}
\label{aima2017}
In this section we compute the Taylor series of 
the integration of differential forms 
over monodromies of the rational curve 
\begin{equation}
\label{CarolinePilar}
 \P^1:  
\left\{
 \begin{array}{l}
 x_{0}-\zeta_1x_{1}=0,\\
 x_{2}-\zeta_2 x_{3}=0,
 \end{array}
 \right.\ \ \ \ \ \zeta_1^d=\zeta_2^d=-1,
 \end{equation}
 inside the Fermat surface $X_0: x_0^d+x_1^d+x_2^d+x_3^d=0$. The content of this section is a reformulation of
 \cite[\S 18.3]{ho13}.
 For a rational number $r$ 
let $[r]$ be the integer part of $r$, that is $[r]\leq r<[r]+1$, and  $\{r\}:=r-[r]$. Let
also $(x)_y:=x(x+1)(x+2)\cdots(x+y-1),\ (x)_0:=1$ be the Pochhammer symbol. For $\beta\in \N_0^{4}$,   $\bar\beta\in \N_0^{4}$ is defined by the rules:
$$
0\leq \bar\beta_i \leq d-1,\ \  \beta_i\equiv_{d}\bar \beta_i.
$$
Consider the family of surfaces in \eqref{15dec2016}. 
\begin{theo}
\label{InLabelNadasht?}
Let $\delta_{t}\in H_2(X_t,\Z),\ t\in(\T,0)$ be the monodromy (parallel transport) of the cycle 
$\delta_0:=[\P^1]\in H_2(X_0,\Z)$ along a path which 
connects $0$ to $t$. 
For a monomial $x_0^{\beta_0} x_1^{\beta_1}x_2^{\beta_2}x_{3}^{\beta_{3}}$ of degree $d\cdot k-4$  
we have 
 \begin{eqnarray}
 \label{15.12.16}
 & & 
 \frac{ -d^{2}  \cdot (k-1)!}{ 2\pi \sqrt{-1}}
 \mathlarger{\mathlarger{\int}}_{\delta_t}\Resi\left(\frac{
 x_0^{\beta_0} x_1^{\beta_1}x_2^{\beta_2}x_{3}^{\beta_{3}}
 \left(  \sum_{i=0}^3 (-1)^ix_i\widehat{dx_i}  \right)
 %\left( 
 %x_0dx_1\wedge dx_2\wedge dx_3-x_1dx_0\wedge dx_2\wedge dx_3+
 %x_2dx_0\wedge dx_1\wedge dx_3-x_3dx_0\wedge dx_1\wedge dx_3
 %\right)
 }{f^{k}_t}\right) =\\
 & & 
\mathlarger{\mathlarger{\mathlarger{\sum}}}_{a: I_d{}\to \N_0}
\left(
\frac{1}{ a! } 
\zeta_1^{  
\overline{(\beta+a^*)_0+1}   }\cdot \zeta_2^{\overline{(\beta+a^*)_2+1}}
\mathlarger{\prod}_{i=0}^{3}\left( \left\{\frac{\check\beta_i+1}{d}\right\}\right)_{ \left[\frac{\check\beta_i+1}{d}\right ]}
 \right) \cdot  t^a,
\nonumber
\end{eqnarray}
where  the sum runs through all $\#I_d$-tuples $a=(a_\alpha,\ \ \alpha\in I_d)$
of non-negative integers such that  
\begin{equation}
\label{TheLastMistake2017}
\left\{\frac{ (\beta+a^*)_{2e}+1}{d} \right\}+   \left\{\frac{ (\beta+a^*)_{2e+1}+1}{d} \right\}=1,\ \ \  
e=0,1, 
\end{equation}
and 
\begin{equation}
t^a:=\prod_{\alpha\in I_d}t_\alpha^{a_\alpha}, \ \ \ \ \  % |a| := \sum_{\alpha\in I_d}a_{\alpha},\\
a!:=\prod_{\alpha\in I_d}a_\alpha!,  \ \   \ \  a^* := \sum_{\alpha}a_\alpha\cdot \alpha.
\end{equation}
\end{theo}

\section{Proof of Theorem \ref{main}, \ref{koshti2019} and \ref{main2}}
We have written a computer code {\tt code1} which for any $C_1$ and $C_2$ as in the Introduction performs the following computations.
For the parameter space \eqref{15dec2016} it uses \eqref{churrasco2019} and \eqref{16.06.2019} in order to compute
the numbers $a_1,a_2,a_3,a_4$ in \eqref{solopeidando2019} for $\cf=11$. In order to be sure that $a_3=\codim(\TS_0V_{[C_1]+\cf[C_2]})$
for generic $\cf\in\Q$, in a separate code called {\tt code3} we have checked this equality for $a_3+2$  values $\cf=2,3,\cdots, a_3+3$.
The number $a_4:=\codim(\TS_0V_{[C_1]}\cap \TS_0V_{[C_2]})$ is the rank of 
vertical concatenation of the matrices $[\pn_{i+j}([C_1])]$ and $[\pn_{i+j}([C_2])]$.
Therefore,  there is no inclusion between  $\TS_0V_{[C_1]}$ and  $\TS_0V_{[C_2]}$ if and only if $a_3\not=a_4$. The main code {\tt code1} 
verifies whether 
$a_3=\binom{d-1}{3} \ \  \& \ \  a_3\not= a_4$. This is the hypothesis of Proposition \ref{danieltoulouse2019}, and so, if the mentioned condition is satisfied  we get a general pencil. All theses cases are gathered under the column `General' in Table \ref{jinbabak2019}. Then {\tt code1} verifies  whether $a_3=a_4$. These are collected under the column `Inclusion'. The rest of the cases are $a_3<\binom{d-1}{3}\ \ \& \ \  a_3\not= a_4$.
The verification of smoothness of $V_{[C_1]+\cf[C_2]}$ is an infinite number of polynomial equalities which at the present moment, the author does no know how to perform it by computer. However, the smoothness of $V_{[C_1]+\cf[C_2]}^N$ (equivalently $N$-smoothness of 
$V_{[C_1]+\cf[C_2]}$) is a finite number of polynomial equalities, see \cite[after Theorem 18.9]{ho13}. For the parameter space \eqref{15dec2016}, this verification is heavy even for $N=2$. 
\begin{rem}\rm
The case $d=5$ is the only instance   in which the author was able to compute $N$-smoothness for $N=2,3,4$ for the family
\eqref{15dec2016}. In this case, it turns out that there are $10$ Noether-Lefschetz locus $V_{[C_1]+\cf[C_2]}$ which are $2$-smooth but not $3$-smooth for all $\cf\in\Q$ as in Theorem \ref{main}. Note that according to Table \ref{jinbabak2019} for the smaller parameter space \eqref{clarakaren2019} and for five of these $10$ case we have to compute until $6$-smoothness.  
\end{rem}
For the rest of the computation we use the parameter space 
$\check\T$ in \eqref{clarakaren2019}. We have verified  that $V_{[C_1]+\cf[C_2]}$ with $\cf$ as 
in \eqref{alexis2019-2} is transversal in $\check\T$ at $0$, that is, the codimension of 
$\TS_0V_{[C_1]+\cf[C_2]}$ in $\TS_0\T$ is equal to the codimension of 
$\TS_0V_{[C_1]+\cf[C_2]}\cap \TS_0\check\T$ in $\T_0\check\T$. All the non-transversal coprime pairs $(\cf_1,\cf_2)$ are collected in the column `NT' of Table \ref{jinbabak2019}. It turns out that such bad cases are included in the set $0\leq \cf_1\leq 2,\ |\cf_2|\le 2$, and that is why we have excluded them in \eqref{alexis2019}. 
For the preparation of this column we have used {\tt code3}.
\begin{rem}\rm
For fixed  $d$ we can make the set \eqref{alexis2019-2} bigger by looking the corresponding data under `NT'. 
For instance, for $d=8$ we only need to exclude the cases $\cf_1, |\cf_2|=0,1$. 
\end{rem}
The transversality statement as above  implies that $V_{[C_1]+\cf[C_2]}^N$ is not smooth in $(\T,0)$ if its scheme theoretical intersection with $\check\T^N$ is not  
smooth. 
%Using {\tt code 3} we have collected the list of all coprime pairs $(\cf_1,\cf_2)$ as in \eqref{alexis2019} such that this condition is not satisfied. It turns out that such bad case are included in the bigger set $1 \cf_1\leq 2 \ \ -2\leq |\cf_2|\leq 2$. 
It follows that if  $V_{[C_1]+\cf[C_2]}\cap \check\T$ is not smooth then $V_{[C_1]+\cf[C_2]}$ is not smooth too. 
Such non-smooth cases are gathered under the columns $N=2,3,\cdots$.
\begin{rem}\rm
 For the exceptional case in Theorem \ref{koshti2019} the only non-transversal case is $(\cf_1,\cf_2)=(1,0)$. In this case $V_{[C_1]+\cf [C_2]}=V_{[C_1]}$ which is a branch of the Noether-Lefschetz locus parameterizing surfaces containing a line. It is known that  $V_{[C_1]}$ for the parameter space \eqref{15dec2016} is smooth.  
\end{rem}

\begin{rem}\rm
 The most time consuming verification has been the proof of Theorem \ref{koshti2019} using {\tt code1}. It took more than $10$ days which is mainly due to implementation of the Taylor series in \S\ref{aima2017}. The verification of $3873$ cases for $d=11$ in Table \ref{jinbabak2019} and using 
 {\tt code2} and the fact that there is no general pencil in this case has taken several days, for further computational details see \S\ref{27.06.2019}.   
\end{rem}

\begin{rem}\rm
 There are $7$ exceptional cases in Table \ref{jinbabak2019}, $d=6, N=2$ and $(d_1,d_2,s_1,s_2,m_1,m_2)$ being:
\begin{eqnarray}
\label{marquesdeabrante2019-1}
 & &(1,2,1,2,1,1), (2,2,2,2,1,2), (2,2,2,2,2,1), \\ 
 \label{marquesdeabrante2019-2}
 & & 
 (1,3,1,3,1,1),(2,3,2,3,2,1),\\
 \label{marquesdeabrante2019-3}
 & & (1,3,1,3,1,2), (2,3,2,3,2,2),
\end{eqnarray}
which are among the $79$ cases (that is why it is stared). 
 For these cases $V^2_{[C_1]+\cf [C_2]}$ the $2$-jet of a smooth  variety  except for $\cf=-1$ in \eqref{marquesdeabrante2019-1},
 for $\cf=-1,\pm \frac{1}{2}$ in \eqref{marquesdeabrante2019-2} and $\cf=1,\pm\frac{1}{2}$ in \eqref{marquesdeabrante2019-3} for both 
 parameter spaces \eqref{15dec2016} and \eqref{clarakaren2019}. 
 Therefore, we have to check the 3-smoothness. It turns out that 
 $V^3_{[C_1]+\cf [C_2]}$ in the parameter space \eqref{clarakaren2019} the $3$-jet of a smooth  variety only for $\cf=0$. 
\end{rem}

\begin{rem}\rm
 For all the cases in Table \ref{jinbabak2019} under the columns $N=2,3,\ldots$, except \eqref{forastf!!}, 
 $V_{[C_1]+\cf[C_2]}$ is  possibly smooth, and hence a special component, for  $\cf=0, \ \pm 1$. These cases are not the focus of this paper, as they give at most  a finite number of special components.  
\end{rem}

\section{Higher dimensions}
All the methods introduced in the present article can be used to investigate the Hodge locus for  the full family $\X/\T$ 
of smooth hypersurfaces of degree $d$ and even dimension $n$.  A Hodge locus $V_{\delta}$ is called general (resp. very general) if its codimension 
(codimension of its Zariski tangent space at a point) is the minimum of 
 the Hodge number $h^{\frac{n}{2}+1, \frac{n}{2}-1}$ and the dimension of the moduli space of hypersurfaces 
 $r:=\binom{n+1-d}{d}-(n+2)^2$.
%Note that for $n=2$ such a minimum is $h^{20}:=\binom{d-1}{3}$. 
For  
\begin{equation}
 \label{paissandu200m}
(n,d)=(2,d), \ \ d\geq 4,\ \  (4,3), (4,4), (4,5), (6,3), (8,3),
\end{equation}
such a minimum is reached by $h^{\frac{n}{2}+1, \frac{n}{2}-1}$ and this is not equal to $r$.
In all these cases the Hodge numbers before $h^{\frac{n}{2}+1, \frac{n}{2}-1}$ are zero, and hence, the number of equations 
defining $V_\delta$ is exactly $h^{\frac{n}{2}+1, \frac{n}{2}-1}$. In these cases very general implies general. 
If the Zariski tangent space $\TS_0V_{\delta}$ has codimension $h^{\frac{n}{2}+1, \frac{n}{2}-1}$, since this is also the number of equations for  
$V_{\delta}$, we conclude that $V_\delta$ is smooth, reduced and general.  The vice versa is not true. 
Take for instance, a sum of two lines intersecting in a point and inside quintic surface. The tangent space at a generic point is of dimension $2d-7=3$, but the Noether-Lefschetz locus in this case is of dimension $h^{20}=2d-6=4$.   
For $(n,d)$ not in \eqref{paissandu200m},  we have $n\geq 4, \ \ d>\frac{2(n+1)}{n-2}$ and $r\leq h^{\frac{n}{2}+1, \frac{n}{2}-1}$. 
A general Hodge cycle by definition satisfies 
\begin{equation}
\label{laguardia-2}
 \codim(\TS_0V_\delta)=\binom{d+n+1}{n+1}-(n+2)^2,
\end{equation}
which is the dimension of the moduli space of hypersurfaces of dimension $n$
 and degree $d$. 
 This implies that a Hodge locus is just a branch of the orbit of $\GL(n+2,\C)$ acting on $0\in\T$. This means that  
 most of the Hodge cycles of the Fermat variety cannot be deformed in the moduli of hypersurfaces, 
 see \cite[\S 16.8]{ho13} for further discussion on this. Proposition \ref{danieltoulouse2019} is valid for arbitrary 
 dimensions replacing $h^{20}$ with $h^{\frac{n}{2}+1,\frac{n}{2}-1}$ in its announcement. 
 
\section{The computer code and data}
\label{27.06.2019}
For the proofs and computation in the present article we have written 
\href{http://w3.impa.br/~hossein/WikiHossein/files/Singular%20Codes/2019_06_Special_Components/2019_06_Code1.txt
}
{{\tt code1}},
\href{http://w3.impa.br/~hossein/WikiHossein/files/Singular%20Codes/2019_06_Special_Components/2019_06_Code2.txt
}
{{\tt code2}}
and
\href{http://w3.impa.br/~hossein/WikiHossein/files/Singular%20Codes/2019_06_Special_Components/2019_06_Code3.txt
}
{{\tt code3}}
which use many procedures from the author's library 
\href{http://w3.impa.br/%7Ehossein/foliation-allversions/foliation.lib
}
{{\tt foliations.lib}}
written in {\sc Singular},  see \cite{GPS01}.
In order to get these codes the reader has two options. 1. The PDF file of the article 
is linked to the author's webpage, and clicking on the name of these codes one gets the corresponding code. 
2. at the bottom of the TEX file of the article in the {\tt arxiv.com} 
one can find the codes. 
In order to check the computations of the present paper,  we 
first get the library 
{\tt foliation.lib}
from \href{http://w3.impa.br/~hossein/mod-fol-exa.html}
{the author's web page. }
\footnote{\tt http://w3.impa.br/$\sim$hossein/foliation-allversions/foliation.lib}
Then we run the codes by doing paste copy, and of course changing the degree $d$ and some other parameters if necessary. 
In order to learn about procedures, for instance {\tt DeformSpace} used in  
\S\ref{02june2019} we run
\begin{verbatim}
  LIB foliation.lib;
  example DeformSapce; 
\end{verbatim}
%Modifying, the code in the example session (for instance changing the dimension $n$ or degree $d$ of the hypersurface)
%we  get all the claimed statements. 
For the computation in this paper we have used  a computer with processor 
{\tt  Intel  Core i7-7700}, $16$ GB Memory plus $170$ GB swap memory and the operating system {\tt Ubuntu 16.04}. 
Note that we have increased the swap memory and this has been very useful for computations of the case $d=8$.

%\bibliography{../Biblio/biblio}
%\bibliographystyle{alpha}

\def\cprime{$'$} \def\cprime{$'$} \def\cprime{$'$} \def\cprime{$'$}

\end{document}